\documentclass{article}

\usepackage[ruled, vlined, onelanguage, noend]{algorithm2e}
\usepackage{amssymb, amsthm}
\usepackage[T1]{fontenc}
\usepackage[hypertexnames=false]{hyperref}
\usepackage{mathtools, mathrsfs, thmtools, tikz-cd}

\DeclareMathOperator{\Aut}{Aut}
\DeclareMathOperator{\coker}{coker}
\DeclareMathOperator{\GL}{GL}
\DeclareMathOperator{\HH}{H}
\DeclareMathOperator{\Hom}{Hom}
\DeclareMathOperator{\id}{id}
\DeclareMathOperator{\im}{im}
\DeclareMathOperator{\lcm}{lcm}
\DeclareMathOperator{\Mat}{Mat}
\DeclareMathOperator{\Poly}{Poly}
\DeclareMathOperator{\pp}{pp}
\DeclareMathOperator{\Span}{Span}
\DeclareMathOperator{\Spec}{Spec}
\DeclareMathOperator{\Supp}{Supp}

\renewcommand{\AA}{\mathbb{A}}
\newcommand{\AO}{\AA^\circ}
\newcommand{\DD}{\mathrm{D}}
\newcommand{\et}{\mathrm{\acute et}}
\newcommand{\FF}{\mathbb{F}}
\newcommand{\Fp}{\FF_p}
\newcommand{\Fq}{\FF_q}
\newcommand{\frakp}{\mathfrak{p}}
\newcommand{\He}{\HH_{\et}}
\newcommand{\Homcr}{\Hom_\mathrm{cr}}
\newcommand{\Hz}{\HH}
\renewcommand{\leq}{\leqslant}
\newcommand{\LL}{\mathscr{L}}
\newcommand{\OO}{\mathcal{O}}
\renewcommand{\Re}{\mathrm{R}\Gamma_{\text{\'et}}}
\newcommand{\RG}{\mathrm{R}\Gamma}
\newcommand{\Wn}{\WW_n}
\newcommand{\WW}{\mathscr{W}}
\newcommand{\xnil}{x_{\rm nil}}
\newcommand{\Xpn}{X^{\langle p^n \rangle}}
\newcommand{\xss}{x_{\rm ss}}
\newcommand{\Ypn}{Y^{\langle p^n \rangle}}
\newcommand{\Zp}{\ZZ/p\ZZ}
\newcommand{\Zpn}{\ZZ/p^n\ZZ}
\newcommand{\ZZ}{\mathbb{Z}}

\newtheorem{cor}{Corollary}[section]
\newtheorem{df}[cor]{Definition}
\newtheorem{lem}[cor]{Lemma}
\newtheorem{notation}[cor]{Notation}
\newtheorem{prop}[cor]{Proposition}
\newtheorem{rk}[cor]{Remark}
\newtheorem{thm}[cor]{Theorem}
\title{Effective Artin--Schreier--Witt theory for curves}
\author{
Christophe Levrat\footnote{Inria Saclay, Palaiseau, France, funded by ANSSI}
\and
Rubén Muñoz-\relax-Bertrand\footnote{Inria \& LIX, CNRS, École polytechnique, Institut Polytechnique de Paris, 91120 Palaiseau, France. This work was supported by the HYPERFORM consortium, funded by France through Bpifrance. The author also worked on this article while at Laboratoire de Mathématiques de Versailles, UVSQ, CNRS, Université Paris-Saclay, 45 avenue des États-Unis, 78035 Versailles Cedex, France and also while at Université Marie et Louis Pasteur, CNRS, LmB (UMR 6623), 25000 Besançon, France where he was under ``Contrat EDGAR-CNRS no 277952 UMR 6623, financé par la région Bourgogne-Franche-Comté.''.}}
\date{\today}

\begin{document}

\maketitle

\begin{abstract}
    We present an algorithm which, given a connected smooth projective curve $X$ over an algebraically closed field of characteristic $p>0$ and its Hasse--Witt matrix, as well as a positive integer $n$, computes all étale Galois covers of $X$ with group $\ZZ/p^n\ZZ$. We compute the complexity of this algorithm when $X$ is defined over a finite field, and provide a complete implementation in \textsc{SageMath}, as well as some explicit examples. We then apply this algorithm to the computation of the cohomology complex of a locally constant sheaf of $\ZZ/p^n\ZZ$-modules on such a curve.
\end{abstract}

\section{Introduction} 
Throughout this article, $p$ shall denote a prime number. Let $X$ be a smooth projective curve over an algebraically closed field $k$ of characteristic $p$. Computing all cyclic étale covers of $X$ of given degree $d$ is an algorithmically difficult task. These covers are parameterised by the étale cohomology group $\He^1(X,\ZZ/d\ZZ)$.

When the degree $d$ is coprime to the characteristic of $k$, one may use the fact that this group is isomorphic to that of $d$-torsion points of the Jacobian $J_X$ of $X$. The corresponding covers, arising by Kummer theory, are built by adjoining to the function field of the curve $d$-th roots of functions whose divisor is a multiple of $d$. Constructing these covers is thus a direct consequence of the computation of the $d$-torsion of $J_X$. Algorithms for doing this have been presented by Huang and Ierardi \cite{huang}, as well as Couveignes \cite{couveignes}, the latter requiring prior knowledge of the zeta function of $X$.

When the degree $d$ is a power $p^n$ of the characteristic of $k$, the aforementioned algorithms do not apply. The case $n=1$ is handled by Artin--Schreier theory, and requires some semilinear algebra. The case $n\geqslant 2$ corresponds to Artin--Schreier--Witt theory and requires the manipulation of Witt vectors. By presenting an algorithm which computes the étale cohomology group $\He^1(X,\ZZ/p^n\ZZ)$ and deduces from its elements the corresponding Artin--Schreier--Witt covers of $X$, we settle the problem of computing all abelian étale covers of $X$ of given degree. 

Our algorithm starts from the data of $X$ as well as a Hasse--Witt matrix for $X$. Computing such a matrix is an algorithmically challenging problem in itself, which has already been the focus of extensive research (see e.g. \cite{kedlaya}, \cite{harvey}, \cite{tuitman}). Computing a basis of $\He^1(X,\Zp)$ from a Hasse--Witt matrix requires finding the fixed points of the (semilinear) Frobenius operator represented by this matrix. 
The algorithmic difficulty in moving from $\He^1(X,\Zp)$ to $\He^1(X,\Zpn)$ essentially lies in Witt vector arithmetic. In our complexity results, we assume that the addition laws on $n$-truncated Witt vectors (which do not depend on the considered curve $X$) have been precomputed; for this, one may use the algorithms presented in \cite{rmb}.

\begin{restatable}{thm}{maincomp}\label{th:maincomp} Let $X$ be a connected smooth projective curve over $\bar\FF_p$, defined over $\Fq$. Suppose we are given a plane model of $X$ of degree $d_X$ with ordinary singularities, and a non-special system of points all defined over $\Fq$. Denote by $g$ the genus of $X$.
Algorithms \ref{alg:H1FromHW} and \ref{alg:MaximalCover} respectively compute $\He^1(X,\Zpn)$ and the maximal abelian étale cover of $X$ of exponent $p^n$ in \[\Poly\left(q^{n+g^2},p^{n^2},d_X\right)\] operations in $\Fq$.  
\end{restatable}

\begin{rk}
The condition concerning the non-special system of points is quite loose. Indeed, as soon as $X$ has $g$ points defined over $\Fq$, such a system exists \cite[Proposition 3.1]{ballet}. This may be achieved by a small base field extension.
\end{rk}

Computing not only the group $\He^1(X,\Zpn)$ but also the corresponding maximal étale Galois covering of exponent $p^n$ allows us to compute more étale cohomology groups on curves, 
namely those of \emph{locally constant} sheaves of finitely generated $(\Zpn)$-modules. 
This is done using methods very similar to those presented in \cite{cl} and only requires a few more algorithmic tricks. 
Considering a curve $X$ obtained by base change from a smaller base field, these groups provide Galois representations that are of interest in their own right.

\begin{restatable}{thm}{etalecomp}\label{th:etalecomp}  Let $X$ be a connected smooth projective curve of genus $g$ over $\bar\FF_p$, defined over $\Fq$. Suppose we are given a plane model of $X$ of degree $d_X$ with ordinary singularities, and a non-special system of $g$ points on $X$ all defined over $\Fq$. Let $\LL$ be a locally constant sheaf of $\Zpn$-modules on $X$, trivialised by a finite étale Galois cover $Y\to X$ of degree $[Y:X]$ defined over $\Fq$. Denote by $m$ the given number of generators of the generic fiber of $\LL$. 
Algorithm \ref{alg:computecohomology} computes the étale cohomology complex of $\LL$  in \[\Poly(q^{n+(g[Y:X])^2},p^{n^2},d_X,m)\] operations in $\Fq$.  
\end{restatable}

We first recall in Section \ref{sec:asw} the main statements of Artin--Schreier--Witt theory which we rely on in the remainder of the article.
We then present in Section \ref{sec:semilinear} an algorithm which computes the fixed points of the Frobenius operator on $\Hz^1(X,\OO_X)$ from a Hasse--Witt matrix of $X$. In all our algorithms, the elements of $\Hz^1(X,\OO_X)$ (resp. $\He^1(X,\Zpn)$) are represented as adeles (resp. Witt vectors of adeles) on $X$. The different ways of computing with these objects are presented in Section \ref{sec:adeles}.
We deduce $\He^1(X,\Zpn)$ from $\He^1(X,\Zp)$ by induction on $n$. The algorithm is summed up in Section \ref{sec:algorithm}, in which we also give an estimate of its complexity.
We have implemented all our algorithms in \textsc{SageMath}. Detailed examples computed using this implementation are presented in Section \ref{sec:implem}. Finally, in Section \ref{sec:etale}, we apply this algorithm to the computation of the cohomology complex of locally constant sheaves of $(\Zpn)$-modules on $X$.

\section{Artin--Schreier--Witt theory}\label{sec:asw}

In this section, we recall the main results of Artin--Schreier--Witt theory, and set some notations for the remainder of the article.
Let $X$ be a connected smooth projective curve over an algebraically closed field of positive characteristic $p$.
Denote by $K$ its function field.
Let $n$ be a positive integer.

\begin{notation}
Given any ring $R$ (resp. sheaf of rings $\mathcal F$ on $X$), we will denote by $W_n(R)$ (resp. $\Wn(\mathcal F)$) the corresponding ring (resp. sheaf) of $p$-typical $n$-truncated Witt vectors.
We denote by $F\colon W_n(R)\to W_n(R)$ (resp. $F\colon \Wn(\mathcal F)\to \Wn(\mathcal F)$) the Frobenius operator, and by $\wp$ the operator $F-\id$.
\end{notation}

\begin{notation}
In the remainder of the article, étale cohomology groups will be denoted by $\He^i$.
Cohomology groups of coherent sheaves for the Zariski topology will be denoted by $\HH^i$. These are actually isomorphic to the cohomology groups of the associated étale sheaf \cite[03DX]{stacks}.
\end{notation}

Artin--Schreier--Witt theory describes the étale Galois covers of $X$ with group $\Zpn$ in terms of Witt vectors.
Here are the main statements that we will use.

\begin{thm}\label{th:asw}
\cite[Hauptsatz I]{schmidwitt} \cite[Proposition 13]{jp}\begin{enumerate}
    \item Given $x\in W_n(K)$ such that no $y\in W_n(K)$ satisfies $\wp(y)=x$, the extension $K(\wp^{-1}(x))$ is an abelian extension of $K$ with group $\Zpn$.
    Any such extension is obtained in this manner.
    \item The group $\He^1(X,\Zpn)$ classifying étale Galois covers of $X$ with group $\Zpn$ is canonically isomorphic to the subgroup of $F$-invariant elements in $\Hz^1(X,\Wn(\OO_X))$.
    \item There is an integer $s_X$ such that for any positive integer $m$, the group $\He^1(X,\ZZ/p^m\ZZ)$ is isomorphic to $(\ZZ/p^m\ZZ)^{s_X}$.
    \item The Galois group of the maximal abelian étale Galois cover of $X$ with exponent $p^n$ is isomorphic to $(\Zpn)^{s_X}$.
\end{enumerate}
\end{thm}

\begin{notation}
We will denote by $\Xpn$ the maximal abelian étale Galois cover of $X$ with exponent $p^n$.
\end{notation}

\section{Computing with semilinear maps}\label{sec:semilinear}

In this section, $R$ will denote a commutative ring, and $\sigma\colon R\to R$ a morphism of rings.
We will describe an effective method to compute the fixed points of a Frobenius-semilinear map.

\subsection{Reminders on semilinear maps}

In this article, we will use the following terminology.

\begin{df}
Let $M$ and $N$ be two $R$-modules.
A \emph{$\sigma$-semilinear map}, or simply a \emph{semilinear map}, is an additive map $F\colon M\to N$ such that:
\begin{equation*}
\forall\lambda\in R,\ \forall m\in M,\ F(\lambda m)=\sigma(\lambda)F(m)\text.
\end{equation*}
\end{df}

In this article, $\sigma$ will always refer to either the Frobenius morphism when $R$ has characteristic $p$, or the $F$ operator on Witt vectors, so that no confusion will arise when we only talk about semilinear maps.

\begin{notation}
Let $M$ and $N$ be free $R$-modules.
Let $(b_i)_{i\in I}$ and $(n_j)_{j\in J}$ be respective $R$-bases of $M$ and $N$ indexed by sets $I$ and $J$.
Let $F\colon M\to N$ be a semilinear map.

We denote by $\Mat_{\mathcal B_M,\mathcal B_N}(F)=(m_{i,j})_{\substack{i\in I \\ j\in J}}$ the unique matrix with coefficients in $R$ such that:
\begin{equation*}
\forall i\in I,\ \forall j\in J,\ F(b_i)=\sum_{j\in J}m_{i,j}n_j\text.
\end{equation*}

When there can be no confusion on the choices of the bases, we will simply denote this matrix by $\Mat(F)$.
\end{notation}

By definition, a semilinear map on free $R$-modules is uniquely determined by its matrix for such $R$-bases.
Indeed, one immediately checks that:
\begin{equation}\label{eq:semilinear}
\forall(\lambda_i)_i\in R^I,\ F\left(\sum_{i\in I}\lambda_ib_i\right)=\begin{pmatrix}\ldots&n_{j\in J}&\ldots\end{pmatrix}\Mat(F)\begin{pmatrix}\vdots\\\lambda_{i\in I}\\\vdots\end{pmatrix}^{(\sigma)}\text,
\end{equation}
where the notation $\bullet^{(\sigma)}$ means that we have applied $\sigma$ to every coefficient of the matrix.

We are interested in the case where $N=M$, and in the fixed points of such semilinear maps.
Denote by $R^{\sigma=\id}$ the subring of $R$ whose elements are the fixed points under $\sigma$.
Denote by $M^{F=\id}$ the set of fixed points under $F$ of $M$.
It is a sub-$R^{\sigma=\id}$-module of $M$, seen as an $R^{\sigma=\id}$-module by restriction of scalars.

\begin{lem}\label{lem:independence}
Assume that $R$ is a field.
Let $M$ be an $R$-vector space, and let $F\colon M\to M$ be a $\sigma$-semilinear map.
Any set of nonzero $R^{\sigma=\id}$-linearly independent elements in $M$ that are fixed points of $F$ is $R$-linearly independent.
\end{lem}

\begin{proof}
Let $(b_i)_{i\in S}$ be such a family of nonzero $R^{\sigma=\id}$-linearly independent fixed points of $F$, where $S$ is a finite set.
Let $(\lambda_i)_{i\in S}$ be a family of scalars in $R$ such that $\sum_{i\in S}\lambda_ib_i=0$.
Assume that this family has the smallest positive number of nonzero elements.
Without loss of generality, we can assume that $\lambda_j=-1$ for some $j\in S$.

Then, $b_j=\sum_{i\in S\smallsetminus\{j\}}\lambda_ib_i$.
Because the $b_i$ are fixed points of $F$, we also get $b_j=\sum_{i\in S\smallsetminus\{j\}}\sigma(\lambda_i)b_i$.
In particular, $\sum_{i\in S\smallsetminus\{j\}}(\sigma(\lambda_i)-\lambda_i)b_i=0$.
By hypothesis on the family, we must have $\sigma(\lambda_i)=\lambda_i$ for every $i\neq j$ in $S$.
In particular, all these $\lambda_i$ belong to $R^{\sigma=\id}$, and by hypothesis they must be naught.
Thus, $b_j=0$, which is impossible.
\end{proof}

In the situation of the above lemma, $R^{\sigma=\id}$ is also a field.
In particular, the canonical $R^{\sigma=\id}$-linear map $M^{F=\id}\otimes_{R^{\sigma=\id}}R\to M$ is an injective $R$-linear map.
This implies that:
\begin{equation*}\label{eq:dimfixedpoints}
\dim_{R^{\sigma=\id}}(M^{F=\id})\leqslant\dim_R(M)\text{.}
\end{equation*}

Notice that $F-\id$ is an $R^{\sigma=\id}$-linear map of $R^{\sigma=\id}$-vector spaces.
So when $\dim_R(M)$ and $R^{\sigma=\id}\to R$ are both finite, computing an $R^{\sigma=\id}$-basis of $M^{F=\id}$ is simple linear algebra.
We have to work a little more, however, when $\dim_R(M)$ is finite but $R^{\sigma=\id}\to R$ is not.

\begin{lem}\label{lem:decomp}
Let $k$ be a field, let $\sigma\colon k\to k$ be a field automorphism and let $M$ be a finite dimensional $k$-vector space.
Let $F$ be a $\sigma$-semilinear map. There exist $F$-stable subspaces $N,S$ of $M$ such that $M=N\oplus S$, that $F|_N$ is nilpotent and $F|_S$ is invertible.
\end{lem}

\begin{proof}
Applying \eqref{eq:semilinear}, we see that $\im(F)$ and $\ker(F)$ are sub-$k$-vector spaces of $M$ since $\sigma$ is an automorphism.
We thus get a decreasing sequence of $k$-vector spaces $(F^i(M))_{i\in\mathbb N}$ which is eventually constant.

Let us denote by $j\in\mathbb N$ the integer at which the sequence stabilises.
Then $M=\im(F^j)\oplus\ker(F^j)$.
Moreover, $F|_{\ker(F^j)}$ is nilpotent, in particular all the fixed points of $F$ lie in $\im(F^j)$.
Furthermore, for any $R$-basis of $\im(F^j)$, the representative matrix of $F|_{\im(F^j)}$ is invertible.
\end{proof}

As the subspaces $N$ and $S$ can be computed using standard linear algebra algorithms, we will always assume that the representative matrix of $F$ is either nilpotent or invertible.

\subsection{Computing fixed points of semilinear maps}

We now assume that $k$ is an algebraic closure of $\FF_p$, and that $M$ is a finite dimensional $k$-vector space.
For a fixed power $q$ of $p$, we set: \[ \sigma\colon\begin{array}{rcl} k&\to&k\\x&\mapsto&x^q\end{array}\]
We shall explain how to effectively compute the fixed points of a $\sigma$-semilinear map $F\colon M\to M$.

It follows from Lemma \ref{lem:decomp} that we can assume without loss of generality that the representative matrix of $F$ for any $k$-basis of $M$ is invertible.

\begin{prop}\label{prop:dieudo}
Under the above assumptions, there exists a $k$-basis $\mathcal B$ of $M$ such that $\Mat_{\mathcal B}(F)$ is the identity matrix.
In other words, the elements of $\mathcal B$ are fixed points of $F$.
Furthermore, $\Span_{\FF_q}(\mathcal B)$ is the set of all fixed points of $F$.
\end{prop}

\begin{proof}
The first part of the statement is proven in \cite[proposition 5]{dieudo} in the case where $\sigma$ is the Frobenius morphism, but the proof strategy holds in our setting too.
A more general version of this statement was later proven by Serge Lang in the context of algebraic groups; see also \cite[main theorem]{dempetal}.

For the needs of our algorithm, we present here a slightly different proof which is constructive.

Let $(f_i)_{i\in I}$ be a set of $k$-linearly independent fixed points of $F$, where $I$ is a possibly empty set of cardinal smaller than $\dim_k(M)$.
Let $a\in M\smallsetminus\Span_k((f_i)_{i\in I})$.
Let $j\in\mathbb N^*$ be the smallest integer such that the family $(f_i)_{i\in I}\cup(F^l(a))_{l=0}^j$ is not $k$-linearly independent, and denote by $N$ the $k$-vector space spanned by this family.
Then, $N$ is stable under $F$.

Let $(\lambda_i)_{i\in I\cup\{0,\ldots,j-1\}}$ be scalars in $k$ such that:
\begin{equation*}
F^j(a)=\sum_{i\in I}\lambda_if_i+\sum_{l=0}^{j-1}\lambda_lF^l(a)\text{.}
\end{equation*}

Our aim is to find a $k$-basis of $N$ whose elements are fixed points of $F$.
Let $(\alpha_i)_{i\in I\cup\{0,\ldots,j-1\}}$ be scalars in $k$ such that $\sum_{i\in I}\alpha_if_i+\sum_{l=0}^{j-1}\alpha_lF^l(a)$ is a fixed point of $F$.
In other words, we have:
\begin{multline*}
\sum_{i\in I}\alpha_if_i+\sum_{l=0}^{j-1}\alpha_lF^l(a)\\=\sum_{i\in I}{\alpha_i}^qf_i+\sum_{l=1}^{j-1}{\alpha_{l-1}}^qF^l(a)+{\alpha_{j-1}}^q\sum_{i\in I}\lambda_if_i+{\alpha_{j-1}}^q\sum_{l=0}^{j-1}\lambda_lF^l(a)\text{.}
\end{multline*}

This yields the following system of equations:
\[
\left\lbrace\begin{aligned}
\forall i\in I,\ \alpha_i&={\alpha_i}^q+{\alpha_{j-1}}^q\lambda_i\\
\alpha_0&={\alpha_{j-1}}^q\lambda_0\\
\forall l\in\{1,\ldots,j-1\},\ \alpha_l&={\alpha_{l-1}}^q+{\alpha_{j-1}}^q\lambda_l
\end{aligned}\right.\]

Which is equivalent to:
\[
\left\lbrace\begin{aligned}
\forall i\in I,\ \alpha_i-{\alpha_i}^q-{\alpha_{j-1}}^q\lambda_i&=0\\
\alpha_0&={\alpha_{j-1}}^q\lambda_0\\
\forall l\in\{1,\ldots,j-2\},\ \alpha_l&={\alpha_{l-1}}^q+{\alpha_{j-1}}^q\lambda_l\\
\alpha_{j-1}-\sum_{l=0}^{j-1}{\lambda_l}^{q^{j-l-1}}{\alpha_{j-1}}^{q^{j-l}}&=0
\end{aligned}\right.\]

On the last line we recognise a $q$-polynomial in $\alpha_{j-1}$, also called a linearised polynomial.
The set of its roots is an $\FF_q$-vector space of dimension $j$ because its derivative is $1$, and they uniquely determine all of the $\alpha_l$ for $l\in\{0,\ldots,j-1\}$.
For $i\in I$, the first line also uniquely determines up to addition in $\FF_q$ the corresponding $\alpha_i$.

Thus, the aforementioned $\FF_q$-vector space of roots yields an $\FF_q$-linearly independent set of fixed points of $F$ of dimension $j$, which is also $\FF_q$-linearly independent from $(f_i)_{i\in I}$.
By lemma \ref{lem:independence}, this set is also $k$-linearly independent, and we have thus constructed a $k$-basis of $N$ of fixed points.

We can repeat this process to get such a basis for $M$.
\end{proof}

Algorithm \ref{alg:fixedpoints} follows this proof.

\begin{algorithm}[H]\label{alg:fixedpoints}
\SetAlgoLined  
\caption{\textsc{FixedPoints}}
\KwData{$d$-dimensional $k$-vector space $M$

Basis $B=(b_i)_{1\leq i\leq d}$ of $M$

$\sigma$-semilinear map $F\colon M\to M$ given by its matrix in the basis $B$}
\KwResult{$k$-basis of $M$ of fixed points under $F$}
\hrulefill

Set $\mathcal B\coloneqq\emptyset$

\For{$i\in\{1,\ldots,\dim_k(M)\}$}{

\If{$b_i\in\Span(\mathcal B)$}{
Continue for loop}

Set $a_0\coloneqq b_i$

Set $a_1\coloneqq F(b_i)$

Set $j\coloneqq1$

\While{$a_j\notin\Span(\mathcal B,a_1,\ldots,a_{j-1})$}{
Set $j\coloneqq j+1$

Set $a_j\coloneqq F(a_{j-1})$}

Find $(\lambda_i)_i$ in $k^{\#\mathcal B+j}$ such that $\lambda_j=\sum_{b\in\mathcal B}\lambda_bb+\sum_{l=0}^{j-1}\lambda_lF(a_l)$

Compute an $\FF_q$-basis $\mathcal R$ of the roots of $X-\sum_{l=0}^{j-1}{\lambda_l}^{q^{j-l-1}}X^{q^{j-l}}$

\For{$r\in\mathcal R$}{
Set $\alpha_{j-1}\coloneqq r$

Set $\alpha_0\coloneqq{\alpha_{j-1}}^q\lambda_0$

\For{$l\in\{1,\ldots,j-2\}$}{
Set $\alpha_l\coloneqq{\alpha_{l-1}}^q+{\alpha_{j-1}}^q\lambda_l$}

\For{$b\in\mathcal B$}{
Compute a root of $X-X^q-{\alpha_{j-1}}^q\lambda_i$ and store it in $\alpha_b$}

Set $\mathcal B\coloneqq\mathcal B\cup\{\sum_{b\in\mathcal B}\alpha_bb+\sum_{l=0}^{j-1}\alpha_la_l\}$}}

\Return{$\mathcal B$}
\end{algorithm}

\begin{lem}\label{lem:compfixedpoints}
Suppose $F$ is defined by a matrix with coefficients in $\Fq$. Set $Q=q^{|\GL_d(\Fq)|}$.
Algorithm \ref{alg:fixedpoints} returns vectors whose coordinates in the given basis $B$ lie in a subfield of $\FF_Q$ and
requires $\tilde O(q^{d^2})$ operations in $\Fq$, where $\tilde O$ is the asymptotic soft-O Landau notation with respect to the parameter $q$.  
\end{lem}

\begin{proof} 
In order to find all of $\mathcal{R}$, we need to perform linear algebra in the splitting field of the given $q$-polynomial.
Since the $q$-degree of this polynomial is bounded by $d$, the Galois group of this extension is a subgroup of $\GL_{d}(\Fq)$ \cite[Lemma 1]{gow}.
Hence, it is a subfield of $\FF_Q$. Every step of the algorithm consists in performing linear algebra operations over $\FF_Q$, which requires $d^\omega \tilde O(\log_q(Q))=\tilde O(q^{d^2})$ operations, where $\omega$ denotes the exponent of matrix multiplication.
\end{proof}

\subsection{Solving the associated inhomogeneous equation}

We assume in this section that $k$ is a field, and that $\sigma\colon k\to k$ is a field automorphism.
We let $M$ be a finite dimensional $k$-vector space, and consider a $\sigma$-semilinear map $F\colon M\to M$.
In this section we are interested in solving in $M$, given some $m\in M$, the equation $F(x)-x=m$.

In this context, Lemma \ref{lem:decomp} ensures that we have a decomposition $M=N\oplus S$ as $k$-vector spaces, such that $F|_N$ is nilpotent, and that $F|_S$ has an invertible representative matrix for any $k$-basis of $S$.
It is therefore enough to give an algorithm to solve $F(x)-x=m$ first in the case where $F$ is nilpotent, then when it has an invertible representative matrix for some $k$-basis of $M$.

So let us assume first that $F$ is nilpotent, and let $n\in\mathbb N$ be the largest integer such that $F^n\neq0$.
Let $x\coloneqq-\sum_{i=0}^nF^i(m)$.
Then:
\begin{equation*}
F(x)-x=\sum_{i=0}^n(-F^{i+1}(m)+F^i(m))=-F^{n+1}(m)+m=m.
\end{equation*}

We now turn to the case where $F$ has an invertible representative matrix for some $k$-basis of $M$.
We furthermore assume the $k$-basis $(b_i)_{i\in I}$ is made of fixed points under $F$, where $I$ is a set of cardinality $\dim_k(M)$.
In the case where $k$ is an algebraic closure of a finite field and $\sigma$ is a power of the Frobenius morphism, Algorithm \ref{alg:fixedpoints} gives us such a $k$-basis.

Under these assumptions, if $(\lambda_i)_{i\in I}$ are scalars in $k$ such that $x\coloneqq\sum_{i\in I}\lambda_ib_i$ satisfies $F(x)-x=m$, we must have $\sigma(\lambda_i)-\lambda_i=m_i$ for all $i\in I$, where  $(m_i)_{i\in I}$ are scalars in $k$ such that $m=\sum_{i\in I}m_ib_i$.
Again, under the assumptions of Algorithm \ref{alg:fixedpoints}, these equations can be solved effectively.

We sum up the above discussion in the following algorithm, assuming that $k$ is an algebraic closure of a finite field and $\sigma$ is a power of the Frobenius morphism.

\begin{algorithm}[H]\label{alg:affeq}
\SetAlgoLined  
\caption{\textsc{InhomEq}}
\KwData{$d$-dimensional $k$-vector space $M$

Basis $B=(b_i)_{1\leq i\leq d}$ of $M$

$\sigma$-semilinear map $F\colon M\to M$ given by its matrix in the basis $B$

Vector $m\in M$ given by its coordinates in the basis $B$}
\KwResult{A solution $x\in M$ of the equation $F(x)-x=m$}
\hrulefill 

Compute $F$-stable suspaces $N,S\subset M$ as in Lemma \ref{lem:decomp}
where $F|_N$ is nilpotent and $F|_S$ is bijective

Set $\xnil\coloneqq0$

Set $n\coloneqq m|_N$

\While{$n\neq0$}{
Set $\xnil=\xnil-n$

Set $n=F(n)$}

Let $\mathcal B$ be any $k$-basis of $S$

Set $\mathcal F\coloneqq\textsc{FixedPoints}(M,\mathcal B,F|_S)$.

Compute $(m_f)_{f\in\mathcal F}\in R^{\mathcal F}$ such that $m|_S=\sum_{f\in\mathcal F}m_ff$

Set $\xss\coloneqq0$

\For{$f\in\mathcal F$}{
Compute $\lambda_f$ a solution of $X^q-X=m_f$

Set $\xss=\xss+\lambda_ff$}

\Return{$\xnil+\xss$}
\end{algorithm}

\begin{lem}\label{lem:inhomeq}
Denote by $\FF_{q^a}$ the smallest extension of $\Fq$ containing all the coordinates of $m$ in the basis $B$. Set $D=q\lcm(a,|\GL_d(\Fq)|)$. 
Algorithm \ref{alg:affeq} returns a vector whose coordinates in the basis $B$ lie in $\FF_{q^D}$ and
requires $\tilde O(D)$ operations in $\Fq$.
\end{lem}

\begin{proof}
By Lemma \ref{lem:compfixedpoints}, the coordinates in $B$ of the elements of $\mathcal F$ lie in $\FF_{q^{|GL_d(\Fq)|}}$.
The coordinates of $m|_S$ in $\mathcal F$ lie in the compositum of this extension of $\Fq$ with the field of definition of the coordinates of $m$ in $B$; the degree of the resulting extension is $\lcm(a,|\GL_d(\Fq)|)$.
The last step of the algorithm requires moving to the degree $q$ extension of this field, which is $\FF_{q^D}$.
The complexity of this algorithm is dominated by the cost of linear algebra computations in $\FF_{q^D}$, which require $d^\omega\tilde O(D)=\tilde O(D)$ operations in $\Fq$.
\end{proof}

\section{Computing with adeles}\label{sec:adeles}
We still consider a smooth projective irreducible curve $X$ over an algebraically closed field $k$ of characteristic $p>0$. We denote by $|X|$ the set of closed points of $X$, and by $K$ its function field.

\subsection{Adeles}

We denote by \[ \AA_X=\left\{ (r_\frakp)_{\frakp}\in\prod_{\frakp\in|X|}K\mid r_\frakp\in\mathcal O_{X,\frakp}\text{ for all but a finite number of }\frakp \right\} \]
the ring of adeles of $X$, and consider its subring of everywhere regular adeles
\[ \AO_X=\prod_{\frakp\in|X|}\OO_{X,\frakp}. \]
The \emph{support} of an adele $r=(r_\frakp)_{\frakp\in |X|}$ is the finite set \[\Supp(r)=\{ \frakp\in |X|\mid r_\frakp\not\in \OO_{X,\frakp}\}.\]
Denote by $\underline{K}$ the constant sheaf associated to the $k$-vector space $K$. The short exact sequence of coherent sheaves \[ 0\to \OO_X \to \underline{K} \to \underline{K}/\OO_X\to 0 \]
yields the following isomorphism of $g$-dimensional $k$-vector spaces \cite[§8]{jp}:
\[ \Hz^1(X,\OO_X)\overset{\sim}{\longrightarrow} \AA_X/(\AO_X+K).\]
This explicit description of the first cohomology group of $X$ will allow us to easily compute with its elements. 

\begin{notation}\begin{itemize}
    \item Given a closed point $\frakp$ of $X$, we will denote by $\delta_\frakp$ the adele whose value is $0$ everywhere, except at $\frakp$ where it is $1$.
    \item Given an adele $r=(r_\frakp)_{\frakp\in |X|}$ and a point $\frakp\in |X|$, we denote by $v_\frakp(r)$ the valuation at $\frakp$ of the function $r_\frakp$.
    \item Let $\frakp$ be a closed point of $X$, and $t$ a uniformiser of the local ring $\OO_{X,\frakp}$. Let $r\in\AA_X$ be an adele. We may write the Laurent series expansion \[ r_\frakp=\sum_{i\geqslant v_\frakp(r)}c_i(r)t^i\] in the completion of the local ring $\OO_{X,\frakp}$. We denote by $\pp_{\frakp,t}(r)$ its principal part, i.e. the tuple $(c_{v_\frakp(r)},\dots,c_{-1}(r))\in k^{\min(0,-v_\frakp(r))}$. Given any integer $s\geqslant \min(0,-v_\frakp(r))$, we will sometimes abuse this notation by still writing $\pp_{\frakp,t_\frakp}(r)$ for the tuple $(0,\dots,0,c_{v_\frakp(r)},\dots,c_{-1}(r))\in k^s$.
\end{itemize}
\end{notation}

\begin{rk}
    Consider two adeles $r,r'\in\AA_X$. The classes of $r$ and $r'$ in $\Hz^1(X,\OO_X)$ are equal if and only if there is a function $h\in K$ such that for any point $\frakp\in |X|$ and any uniformiser $t_\frakp$ at $\frakp$, we have $\pp_{\frakp,t_\frakp}(r)=\pp_{\frakp,t_\frakp}(r'+h)$. The poles of such a function $h$ necessarily lie in $\Supp(r)\cup\Supp(r')$.
\end{rk}

\begin{rk} The construction of a basis of $\Hz^1(X,\OO_X)$ is generally quite easy. In particular, one may always choose each of the elements of the basis to be an adele whose support is a single point. \begin{itemize}
    \item For any $X$, pick a non-special system of points $(\frakp_1,\dots,\frakp_g)$. This means that the Riemann--Roch space of the divisor $\frakp_1+\dots+\frakp_g$ has dimension 1. Denoting by $t_i$ a uniformiser of $\OO_{X,P_i}$, the classes of the adeles $r_1,\dots,r_g$ defined by \[ r_i=\frac{1}{t_i}\delta_{\frakp_i}\]
    form a basis of $\Hz^1(X,\OO_X)$ \cite[§9]{jp}. Such a system is easily constructed by picking the points at random. Indeed, given $\frakp_1,\dots,\frakp_i$ such that $h^0(X,\OO_X(\frakp_1+\dots+\frakp_i))=1$, all but a finite number of $\frakp_{i+1}$ satisfy $h^0(X,\OO_X(\frakp_1+\dots+\frakp_{i+1}))=1$ \cite[§1, 1.]{hw}. 
    \item If $X$ is a hyperelliptic curve given by an equation of the form \[y^2=f(x)\] with $f$ of odd degree $2g+1$, there is a well-known basis of $\Hz^1(X,\OO_X)$ which is usually used.  Denoting by $\infty$ the point at infinity of the curve, this basis is 
\[ \left(\frac{y}{x}\delta_\infty,\frac{y}{x^2}\delta_\infty,\dots,\frac{y}{x^g}\delta_\infty\right).\]
It is (up to scalar multiplication) the dual basis of the usual basis
of $\Hz^0(X,\OO_X)$ given by \[ \left(\frac{dx}{y},x\frac{dx}{y},\dots, x^{g-1}\frac{dx}{y}\right)\]
for the Serre duality pairing.
\end{itemize}
\end{rk}

\begin{rk} Our algorithms take a Hasse--Witt matrix of $X$ as input. There are a great number of algorithms computing a Hasse--Witt matrix for $X$, i.e. the matrix of the Frobenius operator on $\Hz^1(X,\OO_X)$ in a given basis. Methods based on Kedlaya's algorithm \cite{kedlaya}, such as that of Tuitman \cite{tuitman}, compute a Hasse--Witt matrix of any curve (given a smooth lift to characteristic zero) defined over $\FF_{p^\alpha}$ by an equation of degree $d$ in time $\Poly(p,d,\alpha)$. There are also algorithms which run in average polynomial time in $\log(p)$ for hyperelliptic curves \cite{harvey} and plane quartics \cite{costa}. 
\end{rk}

\subsection{Computing in \texorpdfstring{$\Hz^1(X,\OO_X)$}{H1O} and \texorpdfstring{$\He^1(X,\ZZ/p\ZZ)$}{H1Zp}}

Algorithmically speaking, we only consider the equivalence classes of adeles up to everywhere regular adeles. The class of an adele $r$ is then given by the list of the points in the support of $r$ as well as, for each $\frakp\in \Supp(r)$, the function $r_\frakp\in K$.

Let $S$ be a finite set of closed points of $X$. For each $\frakp\in S$, consider a uniformiser $t_\frakp$ at $\frakp$. Let $m=(m_\frakp)_{\frakp\in S}\in (\ZZ_{\leqslant 0})^S$. Define the linear map:
\begin{equation*}\Phi_{S,m}\colon \begin{array}{rcl}
\Hz^0\left(X,\OO_X\left(-\sum_{\frakp\in S}m_\frakp\frakp\right)\right) &\longrightarrow& \prod_{\frakp\in S}k^{-m_\frakp}\\
h &\longmapsto& (\pp_{\frakp,t_\frakp}(h))_{\frakp\in S}\end{array}
\end{equation*}
Note that for simplicity, we omit to mention the uniformisers $t_\frakp$ in the notation.

\begin{lem}\label{lem:fonctions}  \begin{enumerate}
    \item For any such $S$ and $m$, the kernel of $\Phi_{S,m}$ is the set $k$ of constant functions on $X$.
    
    \item Let $r\in\AA_X$ be an adele with support in $S$. Define $m=(m_\frakp)_{\frakp\in S}$ by $m_\frakp=v_\frakp(r)$. The image of $r$ in $\Hz^1(X,\OO_X)$ is trivial if and only if $(\pp_{\frakp,t_\frakp}(r))_{\frakp \in S}$ lies in the image of $\Phi_{S,m}$.
    
     \item Let $r\in\AA_X$ be an adele with support in $S$. Let $r^{(1)},\dots,r^{(g)}\in\AA_X$ be adeles with support in $S$ whose classes form a basis of $\Hz^1(X,\OO_X)$. 
    For every $\frakp\in S$, fix a uniformiser $t_\frakp$ at $\frakp$ and set $m_\frakp=\min (v_\frakp(r),v_\frakp(r^{(1)}),\dots,v_\frakp(r^{(g)}))$.
    Let $D\coloneqq-\sum_{\frakp\in S}m_\frakp \frakp$.
    The $k$-linear map 
    \begin{equation*}\Psi_r\colon \begin{array}{rcl}
    k^g\times \Hz^0\left(X,\OO_X\left(D\right)\right) & \longrightarrow&  \prod_{\frakp\in S}k^{- m_\frakp}\\
    (\beta,h) & \longmapsto & \left(\pp_{\frakp,t_\frakp}\left(\sum_{j=1}^g\beta_j r^{(j)}_{\frakp}+h\right)\right)_{\frakp\in S}
\end{array}\end{equation*}
has kernel $\{ 0\}\times k$, and $(\pp_{\frakp,t_\frakp}(r))_{\frakp\in S}$ lies in its image.
\end{enumerate}
\end{lem}

\begin{proof}
\begin{enumerate}
    \item This is a direct consequence of the fact that the only functions on $X$ with no poles are the constant functions.
    
    \item  The image of $r$ in $\Hz^1(X,\OO_X)$ is trivial if and only if there is a function $h\in K$ which, at every $\frakp\in\Supp(r)$, satisfies $\pp_{\frakp,t_\frakp}(h)=\pp_{\frakp,t_\frakp}(r)$. 
    If such a function exists, its valuation at each $\frakp\in\Supp(r)$ is exactly that of $r$, hence the function lies in $\Hz^0(X,\OO_X(-\sum_{\frakp\in\Supp(r)}v_\frakp(r)\frakp))$.
    
    \item Notice that $\Psi_r(\beta,h)=\Phi_{S,m}(\sum_j\beta_jr^{(j)}+h)$. 
    Since the classes of the adeles $r^{(1)},\dots,r^{(g)}$ in $\Hz^1(X,\OO_X)$ are $k$-linearly independent, the only $\beta\in k^g$ for which there exists an $h$ such that $\Psi_r(\beta,h)=0$ is 0. 
    Hence $\ker(\Psi_r)=0\times\ker(\Phi_r)=0\times k$. Since the classes of $r^{(1)},\dots,r^{(g)}$ span $\Hz^1(X,\OO_X)$, there exists a $\beta\in k^g$ such that, in $\Hz^1(X,\OO_X)$, $\sum_j \beta_jr^{(j)}=r$.
    This means that there exists a function $h\in K$ such that $\pp_{\frakp,t_\frakp}(\sum_j\beta_jr^{(j)}-r)=\pp_{\frakp,t_\frakp}(h)$ for all $\frakp\in S$. 
    This function necessarily lies in $\Hz^0(X,\OO_X(-\sum_{\frakp\in S}m_\frakp\frakp))$.
\end{enumerate}
\end{proof}

\begin{algorithm}[H]\label{alg:findfunction}
\SetAlgoLined  
\caption{\textsc{FindFunction}}
\KwData{Finite set $S\subset |X|$ and adele $r$ given by $(r_\frakp)_{\frakp\in S}\in K^S$}
\KwResult{A function $h\in K$ such that $r-h\in\AO_X$ if the class of $r$ is trivial in $\Hz^1(X,\OO_X)$, and $\perp$ otherwise}
\hrulefill

\For{$\frakp\in S$}{
Compute $v_\frakp(r)$

Compute $\pp_{\frakp,t_\frakp}(r)$}
Compute basis $B$ of $L\coloneqq \Hz^0(X,\OO_X(-\sum_{\frakp\in S}v_\frakp(r)\frakp))$

Compute matrix of $\Phi_r\colon L\to \prod_{\frakp\in S}k^{-v_\frakp(r)}$ w.r.t. $B$

Compute set ${\rm Sol}_r$ of solutions of linear system $\Phi_r(h)=\pp_{\frakp,t_\frakp}(r)$

\If{${\rm Sol}_r\neq\emptyset$}{
\Return any $h\in {\rm Sol}_r$}
\Else{
\Return $\perp$
}
\end{algorithm}

\begin{rk}
In the following complexity computations, we will frequently use the following well-known results (see for instance \cite{abelard}). Suppose we are given a plane model of $X$ with ordinary singularities, defined by a polynomial of degree $d_X$. Given a closed point $\frakp\in |X|$, a function $f\in k(X)$ whose numerator and denominator have degree at most $d_f$, and a divisor $D=D^+-D^-$ on $X$ where $D^+,D^-$ are effective and of degree at most $d_D$: \begin{itemize}
    \item the valuation or the evaluation of $f$ at $P$ can be computed in $\Poly(d_X,d_f)$ operations in $k$;
    \item the principal part of the Laurent series of $f$ at $P$ can be computed in $\Poly(d_X,d_f)$ operations in $k$;
    \item a basis of the Riemann--Roch space $\Hz^0(X,\OO_X(D))$ can be computed in $\Poly(d_X,d_D)$ operations in $k$, and the degree of the numerator and denominator of the computed basis elements have degree $\Poly(d_X,d_D)$. 
\end{itemize}
\end{rk}

\begin{lem}\label{lem:compfindfunction} Using the notations of Algorithm \ref{alg:findfunction}, set $m=\sum_{\frakp\in S}|v_\frakp(r)|$.
Algorithm \ref{alg:findfunction} requires $\Poly(|S|,m,d_X)$ operations in the field of definition of $r$.  
\end{lem}

\begin{proof}
The algorithm consists in $|S|$ principal part computations, one Riemann--Roch space computation for an effective divisor of degree $m$, as well as $d|S|$ evaluations of functions of valuation at most $m$ and solving one $m\times m$ linear system.
\end{proof}

\begin{algorithm}[H]\label{alg:coordinatesinbasis_$p$}
\SetAlgoLined  
\caption{\textsc{CoordinatesInBasis}}
\KwData{Finite set $S\subset |X|$ and uniformisers $t_\frakp$ at all $\frakp\in S$

Adeles $r_0,r_1,\dots,r_g$ each given by $(r_{i,\frakp})_{\frakp\in S}\in K^S$, such that $r_1,\dots,r_g$ form a basis of $\Hz^1(X,\OO_X)$}
\KwResult{$(\beta,h)\in k^g\times K$ such that $r_0-\sum_j\beta_jr_j-h\in\AO_X$}
\hrulefill 

\For{$\frakp\in S$}{
    \For{$i=0\dots g$}{
        Compute $v_\frakp(r_i)$
    }
    Set $m_\frakp=\min_{0\leq i\leq g}v_\frakp(r_i)$
    
    Compute $\pp_{\frakp,t_\frakp}(r_0)$
}
Compute basis $B$ of $L\coloneqq \Hz^0(X,\OO_X(-\sum_{\frakp\in S}m_\frakp\frakp))$

Compute matrix of $\Psi_r\colon k^g\times L\to \prod_{\frakp\in S}k^{m_\frakp}$ w.r.t. $B$ (see Lemma \ref{lem:fonctions})

Find solution $(\beta,h)$ of linear system $\Psi_r(\beta,h)=(\pp_{\frakp,t_\frakp}(r_0))_{\frakp\in S}$

\Return $(\beta,h)$
\end{algorithm}

\begin{lem}\label{lem:compcoordinatesinbasis} Using the notations of Algorithm \ref{alg:coordinatesinbasis_$p$}, set $m=-\sum_{\frakp\in S}m_\frakp$.
Algorithm \ref{alg:coordinatesinbasis_$p$} requires $\Poly(|S|,m,g,d_X)$ operations in the field of definition of $r_0,\dots,r_g$.  
\end{lem}

\begin{proof}
The algorithm consists in $|S|$ principal part computations, one Riemann--Roch space computation for an effective divisor of degree $m$, as well as $m(m+g)$ evaluations of functions of valuation at most $m$, and solving one linear system of size $m\times (m+g)$.
\end{proof}

\section{Computing with Witt vectors of adeles}

In this section, $n$ denotes a positive integer. We now turn our attention to the representation of elements in the first cohomology group $\Hz^1(X,\Wn(\OO_X))$ of the sheaf of $n$-truncated Witt vectors on $X$. 

\subsection{Witt vectors of adeles}

\begin{lem}\label{Wnconstant}
Let $X$ be a topological space.
Let $R$ be a commutative ring.
Then the constant sheaf $\underline{W_n(R)}$ is isomorphic to $\Wn(\underline R)$.
\end{lem}

\begin{proof}
First, note $\Wn$ that induces an endofunctor of the category of presheaves of rings on $X$ which preserves sheaves.
In particular, if $c(R)$ denotes the constant \emph{presheaf} on $X$ with value $R$, we have a natural morphism $\Wn(c(R))\to \Wn(\underline R)$ of presheaves of rings on $X$.

But the stalks of $\Wn(\underline R)$ are all $W_n(R)$, so the lemma follows.
\end{proof}

The proof of the following proposition follows the lines of the classical proof for $\Hz^1(X,\OO_X)$ which can be found in \cite[§8]{jp}. Since we could not find this particular result in the literature, we give a detailed proof of it below. Recall that we denote by $K$ the function field of $X$.

\begin{prop}\label{lem:H1Wn} Let $n$ be a positive integer. There are canonical isomorphisms of $W_n(k)$-modules:
\begin{align*}
\Hz^1(X,\Wn(\OO_X)) &\overset{\sim}{\longrightarrow} \frac{\displaystyle\bigoplus_{\frakp\in |X|}\frac{W_n(K)}{W_n(\OO_{X,\frakp})}}{W_n(K)} \\
&\overset{\sim}{\longrightarrow
} \frac{W_n(\AA_X)}{W_n(\AO_X)+W_n(K)}.
\end{align*}
\end{prop}

\begin{proof}
Consider the exact sequence of sheaves of $W_n(k)$-modules \[ 0\to \Wn(\OO_X)\to \WW_n(\underline{K}) \to \Wn(\OO_X)/\Wn(\underline{K})\to 0.\]
Since the sheaf $\Wn(\underline{K})$ is constant on the integral curve $X$ by Lemma \ref{Wnconstant}, it is acyclic, so \[\Hz^1(X,\Wn(\OO_X))=\coker(\Hz^0(X,\Wn(\underline{K}))\to \Hz^0(X, \Wn(\underline{K})/\Wn(\OO_X))).\]
For any closed point $\frakp$ of $X$, denote by $i_\frakp\colon \Spec(k)\to X$ the corresponding morphism.
The sheaf ${i_\frakp}^\star (\Wn(\underline{K})/\Wn(\OO_X))$ is simply the $W_n(k)$-module $W_n(K)/W_n(\OO_{X,\frakp})$.
The adjunction map \[ \Wn(\underline{K})/\Wn(\OO_X) \to \bigoplus_{\frakp\in |X|}{i_\frakp}_{\star} {i_\frakp}^\star \left( \Wn(\underline{K})/\Wn(\OO_X)\right)\]
may thus be rewritten as
\begin{equation} \Wn(\underline{K})/\Wn(\OO_X) \to \bigoplus_{\frakp\in |X|}{i_\frakp}_{\star}\left( W_n(K)/W_n(\OO_{X,\frakp})\right).\tag{$\diamond$}\end{equation}
Since $W_n$ commutes with filtered colimits of rings, the stalk of $\Wn(\OO_X)$ at $\frakp$ is $W_n(\OO_{X,\frakp})$.
Hence, the stalk at $\frakp$ of the map $(\diamond)$ is the identity map of $W_n(K)/W_n(\OO_{X,\frakp})$, and $(\diamond)$ is an isomorphism.
Therefore, since $X$ is quasi-compact and quasi-separated, the $W_n(k)$-module of global sections of the quotient sheaf $\Wn(\underline{K})/\Wn(\OO_X)$ is canonically isomorphic to the direct sum of the $W_n(K)/W_n(\OO_{X,\frakp})$ for $\frakp\in |X|$.
This concludes the proof of the first isomorphism.

For the second one, it suffices to notice that the first expression we have just obtained of $\Hz^1(X,\Wn(\OO_X))$ is the same as \[ \frac{\{f\in\prod_{\frakp\in|X|}W_n(K)\mid f_\frakp\in W_n(\OO_{X,\frakp})\text{ for all but a finite number of $\frakp$}\}}{\prod_{\frakp\in|X|}W_n(\OO_{X,\frakp})+W_n(K)}\]
which, since the functor $W_n$ commutes with products, is canonically isomorphic to $W_n(\AA_X)/(W_n(\AO_X)+W_n(K))$.
\end{proof}

\begin{notation}
For $n\geqslant 0$, denote by $S_n\in\ZZ[X_0,\dots,X_n,Y_0,\dots,Y_n]$ the polynomial defining the $n$-th coordinate of the sum of two Witt vectors, and set $R_n=S_n-(X_n+Y_n)\in\ZZ[X_0,\dots,X_{n-1},Y_0,\dots,Y_{n-1}]$. For any Witt vectors $v,w$, we denote by $R_n(v,w)$ the element $R_n(v_0,\dots,v_{n-1},w_0,\dots,w_{n-1})$. Given a Witt vector $r$, we denote by $r_{<n}$ the $n$-truncated Witt vector $(r_0,\dots,r_{n-1})$.
\end{notation}

Algorithm \ref{alg:findfunctionwitt} determines, given an element $r$ of $W_n(\AA_X)$, whether it belongs to $W_n(K)+W_n(\AO_X)$. If it is the case, it returns $h\in W_n(K)$ such that $r-h\in W_n(\AO_X)$. It rests on the following observation. 

\begin{lem} Let $r\in W_{n+1}(\AA_X)$ be a Witt vector of adeles whose class in $\Hz^1(X,\WW_{n+1}(\OO_X))$ is trivial. For any $(h,a)\in W_{n}(K)\times W_{n}(\AO_X)$ such that $ r_{<n}=h+a$ , there exist $(h_n,a_n)\in K\times\AO_X$ such that $r=(h_0,\dots,h_n)+(a_0,\dots,a_n)$ in $W_n(\AA_X)$. Moreover, given any such $(h_n,a_n)$, \[ (r_0,\dots,r_n)-(h_0,\dots,h_{n-1},0)=(a_0,\dots,a_{n-1},a_n+h_n).\]
\end{lem}

\begin{proof}
The first assertion follows directly from chasing the following commutative diagram whose vertical maps are all surjective, and whose lines are exact. The second one is a straightforward computation. \[
\begin{tikzcd}[column sep=13pt]
W_{n+1}(k)\arrow[r,hookrightarrow]\arrow[d,twoheadrightarrow]&W_{n+1}(K)\oplus W_{n+1}(\AO_X)\arrow[r,"c_n"]\arrow[d,twoheadrightarrow]&W_{n+1}(\AA_X)\arrow[r,"\pi_n",twoheadrightarrow]\arrow[d,twoheadrightarrow]&\Hz^1(X,\WW_{n+1}(\OO_X))\arrow[d,twoheadrightarrow]\\
W_{n}(k)\arrow[r,hookrightarrow]&W_{n}(K)\oplus W_{n}(\AO_X)\arrow[r,"c_{n-1}"]&W_{n}(\AA_X)\arrow[r,"\pi_{n-1}",twoheadrightarrow]&\Hz^1(X,\mathcal \WW_{n}(\OO_X))
\end{tikzcd}
\]
\end{proof}

\begin{algorithm}[H]\label{alg:findfunctionwitt}
\SetAlgoLined  
\caption{\textsc{FindFunctionWitt}}
\KwData{Finite set $S\subset |X|$ and Witt vector of adeles $(r_0,\dots,r_{n-1})\in W_n(\AA_X)$, each given by $(r_{i,\frakp})_{\frakp\in S}\in K^S$}
\KwResult{A Witt vector of functions $(h_{0},\dots,h_{n-1})\in W_n(K)$ such that $r-h\in W_n(\AO_X)$ if the class of $r$ is trivial in $\Hz^1(X,\Wn(\OO_X))$, and $\perp$ otherwise}
\hrulefill

\For{$i=0\dots n-1$}{
Compute $u_i=r_i+R_i(r_{<i},-h_{<i})$

Compute $h_i=\textsc{FindFunction}(S,u_i)$

\If{$h_i=\perp$}
{\Return $\perp$}
}
\Return $(h_0,\dots,h_n)$
\end{algorithm}

\begin{rk}
In the following complexity estimates, we will always assume that the polynomials $R_n, S_n$ (which only depend on $n$) defining addition of Witt vectors have been precomputed. For details about how to compute these polynomials, see \cite{rmb}. They have degree $p^n$, hence adding two $n$-truncated Witt vectors in a ring requires $\Poly(n\log(p))$ operations in this ring.
\end{rk}

\begin{lem} Using the notations of Algorithm \ref{alg:findfunctionwitt}, set \[ m=\max_{0\leq i\leq n-1}\sum_{\frakp\in S}|v_\frakp(r_i)|.\]
Algorithm \ref{alg:findfunctionwitt} requires $\Poly(p^{n(n-1)/2},|S|,m,d_X)$ operations in the field of definition of $r_0,\dots,r_{n-1}$.  
\end{lem}

\begin{proof} Algorithm \ref{alg:findfunctionwitt} consists in calls to Algorithm \ref{alg:findfunction} for the adeles $u_0,\dots,u_{n-1}$. 
For all $i\in \{0,\dots,n-1\}$, the total degree of $R_i$ is $p^i$ and a simple induction argument shows that the valuation at any $\frakp\in S$ of $u_i$ is at most $p^{i(i+1)/2}m$. Lemma \ref{lem:compfindfunction} concludes.
\end{proof}

The following algorithm allows, given a Witt vector $r$ of adeles representing an element of $\He^1(X,\Zpn)$ as well as Witt vectors of adeles representing a basis of the free $\Zpn$-module $\He^1(X,\Zpn)$, to compute the coordinates of the class of $r$ in this basis. Here, we use the isomorphism \[ \Zpn\xrightarrow{\sim} W_n(\Zp).\]

\begin{lem}\label{lem:coordsn} Consider a tuple $(b^{(1)},\dots,b^{(s)})\in W_n(\AA_X)^s$ representing a basis $B$ of $\He^1(X,\Zpn)$. For each $i\in \{1\dots s\}$, there exists $h^{(i)}\in W_n(K)$  such that \[F(b^{(i)})-b^{(i)}\equiv h^{(i)}\mod W_n(\AO_X).\]
Let $r\in W_n(\AA_X)$ be a Witt vector representing an element of $\He^1(X,\Zpn)$, and $(\alpha^{(1)},\dots,\alpha^{(s)})\in W_n(\Zp)^s$ be its coordinates in the basis $B$. There exists $h\in W_n(K)$ such that $r-h-\sum_{i=1}^s\alpha^{(i)}b^{(i)}\in W_n(\AO_X)$.
Then the last coordinate of \[ r-\sum_{i=1}^s (\alpha_{<n}^{(i)},0)b^{(i)}-(h_{<n},0)\]
is equal to \[h_n+\sum_{i=1}^s \alpha_n^{(i)}\left(b_0^{(i)}+\sum_{j=0}^{n-1}F^j(h_0^{(i)})\right)\]
modulo $\AO_X$.
\end{lem}

\begin{proof} Set \[ a=r-h-\sum_{i=1}^s\alpha^{(i)}b^{(i)}\]
and 
\[ a^{(i)}=F(b^{(i)})-b^{(i)}-h^{(i)}.\]
Note that \[ F^n(b^{(i)})=b^{(i)}+\sum_{j=1}^{n-1}F^j(h^{(i)}+a^{(i)}).\]
Set $r'=r-\sum_{i=1}^s (\alpha_{<n}^{(i)},0)b^{(i)}-(h_{<n},0)$. Denoting by $V\colon W_n(\AA_X)\to W_n(\AA_X)$ the Verschiebung map, we have
\begin{eqnarray*} r'
&=&\sum_{i=1}^s V^n([\alpha_n^{(i)}])b^{(i)}+V^n([h_n])+a\\
&=& V^n\left(\left[h_n+\sum_{i=1}^s[\alpha_n^{(i)}]F^n(b^{(i)})\right]\right)+a \\
&\equiv& V^n\left(\left[h_n+\sum_{i=1}^s[\alpha_n^{(i)}]\left(b^{(i)} + \sum_{j=0}^{n-1}F^j(h^{(i)}) \right)\right]\right)\mod W_n(\AO_X).
\end{eqnarray*}
\end{proof}

As in the previous algorithms, we adopt a recursive method in order to compute the coordinates of $r$ in the given basis $B$. More pecisely, at the $j$-th iteration of Algorithm \ref{alg:coordsn}, we compute the elements $\alpha^{(i)}_j$ and $h_j$ of Lemma \ref{lem:coordsn}.

\begin{algorithm}[H]\label{alg:coordinatesinbasis_$p^n$}
\SetAlgoLined  
\caption{\textsc{CoordinatesInBasisWitt}}\label{alg:coordsn}
\KwData{Finite set $S\subset |X|$ and uniformisers $t_\frakp$ at all $\frakp\in S$

Witt vector of adeles $r=(r_0,\dots,r_{n-1})\in W_n(\AA_X)$ supported on $S$

Witt vectors of adeles $b^{(0)},\dots,b^{(s)}\in W_n(\AA_X)$ supported on $S$, representing a basis of $\He^1(X,\Zpn)$
}
\KwResult{$h\in W_n(K)$, $\alpha^{(1)},\dots,\alpha^{(s)}\in W_n(\Zp)$ such that $r-h-\sum_{i=1}^s\alpha^{(i)}b^{(i)}\in W_n(\AO_X)$}
\hrulefill 

\For{$i=1\dots s$}{
$h^{(i)}_0\coloneqq \textsc{FindFunction}\left(F(b_0^{(i)})-b_0^{(i)} \right)$

$\left(\left(\alpha^{(1)}_0,\dots,\alpha^{(s)}_0\right),h\right)\coloneqq \textsc{CoordinatesInBasis}\left( r_0, \left(b^{(j)}_0,\dots,b^{(j)}_s\right)\right)$}

\For{$j=1\dots n-1$}{
Compute last coordinate $u_j$ of $(r_0,\dots,r_j)-(h_{<j},0)-\sum_{i=1}^s(\alpha^{(i)}_{<j},0)b^{(i)}$

Compute $\left(\alpha^{(1)}_j,\dots,\alpha^{(s)}_j,h_j\right)\coloneqq\textsc{CoordinatesInBasis}\left( u_j, \left(b_0^{(i)}+\sum_{m=1}^{j-1}F^m(h_0^{(i)})\right)_{1\leq i\leq s}   \right)$ }
\Return $(h_0,\dots,h_n), (\alpha^{(1)}_0,\dots,\alpha^{(1)}_{n-1}),\dots,(\alpha^{(s)}_0,\dots,\alpha^{(s)}_{n-1})$
\end{algorithm}

\begin{lem}\label{lem:compcoordsinbasiswitt} Using the notations of Algorithm \ref{alg:coordinatesinbasis_$p^n$}, we set $T=\{ r_i\}_i\cup \{ b^{(i)}_j\}_{i,j}$,  and \[m=\max_{a\in T}\sum_{\frakp\in S}|v_\frakp(a)|.\]
Algorithm \ref{alg:coordinatesinbasis_$p^n$} requires $\Poly(|S|,s,p^{n(n-1)/2},m,d_X)$ operations in the field of definition of $r,b^{(0)},\dots,b^{(s)}$.  
\end{lem}

\begin{proof}
At step $j\in\{2,\dots,n-1\}$, the valuation of $u_j$ at any $\frakp\in S$ is at most $p^{j}v_\frakp(u_{j-1})\leq p^{j(j+1)/2}m$. The costliest call to Algorithm \ref{alg:coordinatesinbasis_$p$} is the last one, where $u_{n-1}$ has valuation at most $p^{n(n-1)/2}$ at any $\frakp\in S$. Lemma \ref{lem:compcoordinatesinbasis} concludes.
\end{proof}

\subsection{Computing \texorpdfstring{$\He^1(X,\Zpn)$}{H1Zpn} knowing \texorpdfstring{$\He^1(X,\Zp)$}{H1Zp}}

The short exact sequence of abelian étale sheaves on $X$  \[ 0\to \Zpn\to \Wn(\OO_X)\xrightarrow{\wp}\Wn(\OO_X)\to 0\]
yields the following short exact sequence of abelian groups \cite[Proposition 13]{jp}: \[ 0\to \He^1(X,\Zpn)\to \Hz^1(X,\Wn(\OO_X))\to \Hz^1(X,\Wn(\OO_X))\to 0.\]
In particular, this means that $\He^1(X,\Zpn)$ is isomorphic to the subgroup of Frobenius-invariant elements of $\Hz^1(X,\Wn(\OO_X))$.
We are going to use the following natural isomorphism to describe the elements of $\Hz^1(X,\Wn(\OO_X))$, proved in Lemma \ref{lem:H1Wn}:
\[ \Hz^1(X,\Wn(\OO_X)) \xrightarrow{\sim} W_n(\AA_X)/(W_n(\AO_X)+W_n(K)).\]
The computation of $\He^1(X,\Zpn)$ is performed by induction on $n$, using the following result, proved in \cite[Proposition 14, Corollaire]{jp}.

\begin{lem}
The map $\He^1(X,\ZZ/p^{n+1}\ZZ)\to \He^1(X,\Zpn)$, induced under the above isomorphism by the truncation map $W_{n+1}(\AA_X)\to W_{n}(\AA_X)$, is surjective.
\end{lem}

\begin{cor}
Consider a $\Zpn$-basis $(r_1,\dots,r_s)$ of $\He^1(X,\Zpn)$. Let $r'_1,\dots,r'_{s}\in\Hz^1(X,\WW_{n+1}(\OO_X))$ be respective preimages of $r_1,\dots,r_s$ under the map $\Hz^1(X,\WW_{n+1}(\OO_X))\to \Hz^1(X,\Wn(\OO_X))$. Then $(r'_1,\dots,r'_s)$ is a basis of the free $\ZZ/p^{n+1}\ZZ$-module $\He^1(X,\ZZ/p^{n+1}\ZZ)$.
\end{cor}

\begin{proof} Since $\He^1(X,\ZZ/p^{n+1}\ZZ)\to \He^1(X,\Zpn)$ is surjective and its kernel contains $p^n\He^1(X,\ZZ/p^{n+1}\ZZ)$, the map \[ \He^1(X,\ZZ/p^{n+1}\ZZ)\otimes\Zpn\to \He^1(X,\Zpn) \]
is an isomorphism. As the ideal $(p^n)$ of $\ZZ/p^{n+1}\ZZ$ is nilpotent, Nakayama's lemma \cite[07RC, (8)]{stacks} concludes.
\end{proof}

Our recursive strategy for computing $\He^1(X,\Zpn)$ is the following: from a basis of the free $(\ZZ/p^j\ZZ)$-module $\Hz^1(X,\ZZ/p^j\ZZ)$, we compute a preimage of each of these elements in $\He^1(X,\ZZ/p^{j+1}\ZZ)$.
Lemma \ref{lem:eqn+1} makes this problem more explicit.

\begin{notation}
Let $n$ be a positive integer. Consider two tuples of indeterminates $\underline x=(x_0,\dots,x_{n-1})$ and $\underline{y}=(y_0,\dots,y_{n-1})$. We denote by $P_n\in\ZZ[\underline x,\underline y]$ the unique polynomial such that the last component of the Witt vector
\[F(\underline x,0)-(\underline x,0)-(\underline y,0)\in W_{n+1}(\ZZ[\underline x, \underline y])\]
is $P_n(x,y)$. The polynomial $P_n$ has total degree $p^{n+1}$.
\end{notation}

\begin{lem}\label{lem:eqn+1} Consider $r=(r_0,\dots,r_{n})\in W_{n+1}(\AA_X)$ 
and $h=(h_0,\dots,h_n)\in W_{n+1}(K)$ 
such that $\wp(r)-h\in W_{n+1}(\AO_X)$. 
Then \[ r_n^p-r_n\equiv -P_n(r_{<n},h_{<n})\mod \AO_X+K.\]
Conversely, given any $s_n\in\AA_X$ such that \[ s_n^p-s_n\equiv -P_n(r_{<n},h_{<n})\mod \AO_X+K\]
the class of $(r_0,\dots,r_{n-1},s_n)$ in $\Hz^1(X,\WW_{n+1}(\OO_X))$ belongs to $\He^1(X,\ZZ/p^{n+1}\ZZ)$.
\end{lem}

\begin{proof} \ There exists a Witt vector $a\in W_{n+1}(\AO_X)$ such that $\wp(r)=h+a$.
A straightforward computation shows that:
\[\wp(r)-(h_0,\dots,h_{n-1},0)=(a_0,\dots,a_{n-1},r_n^p-r_n+P_n(r_{<n},h_{<n})).\]
Since $\wp(r)=h+a$, we also have
\[ \wp(r)-(h_0,\dots,h_{n-1},0)=(0,\dots,0,h_n)+a\]
which shows that 
\[  r_n^p-r_n=-P_n(r_{<n},h_{<n})+h_n+a_n.\]
Conversely, if there are elements $h'_n\in K$ and $a'_n\in\AO_X$ such that $s_n$ satisfies \[ s_n^p-s_n+P_n(r_{<n},h_{<n})=h'_n+a'_n\] then we have \[ \wp(r_0,\dots,r_{n-1},s_n)=(h_0,\dots,h_{n-1},h'_n)+(a_0,\dots,a_{n-1},a'_n)\] by the same computations as above. 
\end{proof}

 Once the free module $\He^1(X,\Zpn)$ has been computed, the tuples $r_{<n}$ and $h_{<n}$ are known. Hence, Lemma \ref{lem:eqn+1} guarantees that finding a preimage of $r_{<n}$ in $\He^1(X,\ZZ/p^{n+1}\ZZ)$ reduces to finding $r_n$ by solving an equation of the form \[ r_n^p-r_n=v_n\] in $\Hz^1(X,\OO_X)$, where $v_n$ can easily be computed from $r_{<n},h_{<n}$. This is done using Algorithm \ref{alg:affeq}, and yields the following algorithm.

\begin{algorithm}[H]\label{alg:H1pn}
\SetAlgoLined  
\caption{\textsc{ComputeH1}}
\KwData{Function field $K$ of smooth projective curve $X$ over $k$

Finite set $S$ of closed points of $X$

Basis $B$ of $\Hz^1(X,\OO_X)$ given by representatives supported on $S$

Representatives $r_0^{(1)},\dots,r_0^{(s)}\in\AA_X$ (supported on $S$) of an $\FF_p$-basis of $\He^1(X,\Zp)$}
\KwResult{Representatives $r^{(1)},\dots,r^{(s)}\in W_n(\AA_X)$ of a basis of the free $\Zpn$-module $\He^1(X,\Zpn)$

$h^{(1)},\dots,h^{(s)}\in W_n(K)$ s.t. $\forall i\in \{ 1\dots s\}, r^{(i)}-h^{(i)}\in W_n(\AO_X)$}
\hrulefill

\For{$i=1\dots s$}{
    Compute $h_0^{(i)}=\textsc{FindFunction}\left(S,F(r_0^{(i)})-r_0^{(i)}\right)$
    
    \For{$j=1\dots n-1$}{
    Set $r^{(i)}=(r_0^{(i)},\dots,r_{j-1}^{(i)})$
    
    Set $h^{(i)}=(h_0^{(i)},\dots,h_{j-1}^{(i)})$
    
    Compute $v_j^{(i)}\coloneqq -P_j(r^{(i)},h^{(i)})$
    
    $((u_1^{(i)},\dots,u_g^{(i)}),-)\coloneqq\textsc{CoordinatesInBasis}(S,v_j^{(i)},B)$
    
    $r_j^{(i)}=\textsc{InhomEq}(F,B,(u_1^{(i)},\dots,u_g^{(i)}))$
    
    $h_j^{(i)}=\textsc{FindFunction}\left(S,F(r_j^{(i)})-r_j^{(i)}-v_j^{(i)}\right)$
    }
}
\Return{$\left(r^{(1)},\dots, r^{(s)}\right),\left(h^{(1)},\dots,h^{(s)}\right)$}
\end{algorithm}

\begin{lem}\label{lem:compH1} Suppose $X$ is given by a plane model with ordinary singularities defined by a polynomial of degree $d_X$. 
Set $m=p\cdot \max_{1\leqslant i \leqslant s}\sum_{\frakp\in S} v_\frakp(r_0^{(i)})$.
Algorithm \ref{alg:H1pn} requires $\Poly(q^{g^2},p^{n^2},d_X,|S|,m,n)$ operations in $k$. The field of definition of the output has degree $(n+1)|\GL_g(\Fq)|$ over $\Fq$.
\begin{proof} Since $P_j$ has degree $p^{j+1}$, the valuation of $v_j^{(i)}$ is $p^{j+1}$ times the maximum valuation of the entries of $r^{(i)},h^{(i)}$. By induction, this means that the valuation of $v_n^{(i)}$ is bounded from above by $p^{(n+1)(n+2)/2}$. The field of definition of the $r_j^{(i)}$ also increases at each step. When $j=0$, Lemma \ref{lem:compfixedpoints} tells us that they lie in $\FF_{q_0}$ where $\log_q(q_0)=|\GL_g(\Fq)|$. So does $v_1^{(i)}$. Hence by Lemma \ref{lem:inhomeq}, $r_1^{(i)}$ lies in $\FF_{q_1}$ where $\log_q(q_1)=q\log_q(q_0)$. A quick induction shows that the field of definition $\FF_{q_n}$ of $v_n$ satisfies \[ \log_q(q_n)=q^n|\GL_g(\Fq)|=O(q^{n+g^2}).\]
The total complexity follows from Lemma \ref{lem:inhomeq} and Lemma \ref{lem:compcoordinatesinbasis}.
\end{proof}
\end{lem}

\section{Summary of the algorithms and complexity}\label{sec:algorithm}

In this section, we present the two core algorithms of this article. The first one computes, given a Hasse--Witt matrix of a smooth projective curve $X$ over an algebraically closed field of characteristic $p$, a basis of $\Hz^1(X,\Zpn)$. The second one computes the maximal étale abelian cover of $X$ with exponent $p^n$.

\begin{algorithm}[H]\label{alg:H1FromHW}
\SetAlgoLined  
\caption{\textsc{ComputeH1FromHW}}
\KwData{Function field $K$ of smooth projective curve $X$ over $k$

Positive integer $n$

Finite set of places $S\subset |X|$

Adeles $b^{(1)},\dots,b^{(g)}$ supported on $S$ representing a basis $B$ of $\Hz^1(X,\OO_X)$

Hasse--Witt Matrix $HW$ of $X$ with respect to basis $B$
}
\KwResult{Representatives $(r^{(1)},\dots,r^{(s)})\in W_n(\AA_X)^s$ of a basis of the free $\Zpn$-module $\He^1(X,\Zpn)$

$(h^{(1)},\dots,h^{(s)})\in W_n(K)$ such that $\forall i\in \{1\dots s\}$, $r^{(i)}-h^{(i)}\in W_n(\AO_X)$}
\hrulefill

$(r^{(1)}_0,\dots,r^{(s)}_0)\coloneqq$\textsc{FixedPoints}$(B,HW)$

$(r^{(1)},\dots,r^{(s)}),(h^{(1)},\dots,h^{(s)})\coloneqq$\textsc{ComputeH1}$(K,S,r_0^{(1)},\dots,r_0^{(s)})$

\Return{$(r^{(1)},\dots,r^{(s)}),(h^{(1)},\dots,h^{(s)})$}
\end{algorithm}

\begin{algorithm}[H]\label{alg:MaximalCover}
\SetAlgoLined  
\caption{\textsc{ComputeMaximalCover}}
\KwData{Function field $K$ of smooth projective curve $X$ over $k$

Positive integer $n$

Finite set of places $S\subset |X|$

Adeles $b^{(1)},\dots,b^{(g)}$ supported on $S$ representing a basis $B$ of $\Hz^1(X,\OO_X)$

Hasse--Witt Matrix $M$ of $X$ with respect to basis $B$
}
\KwResult{Function field extension $L/K$ corresponding to maximal étale abelian cover of $X$ with exponent $p^n$}
\hrulefill

$(r^{(1)},\dots,r^{(s)}),(h^{(1)},\dots,h^{(s)})\coloneqq$\textsc{ComputeH1FromHW}$(K,n,S,B,M)$

Set $L=K(t_{0}^{(1)},\dots,t_{n-1}^{(1)},\dots,t_{n-1}^{(s)},\dots,t_{n-1}^{(s)})$ 
where for all $i\in\{1\dots s\}$: $\wp(t^{(i)})=h^{(i)}$ as Witt vectors

\Return{L}
\end{algorithm}

\maincomp*

\begin{proof}
This is a direct consequence of Lemma \ref{lem:compfixedpoints} and Lemma \ref{lem:compH1}, in which we may take $|S|=g$ and $m=q$ given the assumptions made in the statement of the theorem.
\end{proof}

\section{Implementation and examples}\label{sec:implem}

We have implemented Algorithm \ref{alg:MaximalCover} using SageMath 10.8.beta1 \cite{sagemath}.
SageMath in turn uses various external libraries.
Computations with $p$-adics are done with FLINT \cite{flint}, computations with polynomials sometimes use Singular \cite{singular} and computations in finite fields use Givaro \cite{givaro} for fields of small cardinality and PARI/GP \cite{pari} otherwise. Our implementation is available at:
\begin{center}
\small\url{https://rubenmunozbertrand.pages.math.cnrs.fr/artinschreierwitt.py}.
\end{center}

This enabled us to compute the following examples on a Intel Core Ultra 7 165H processor with Debian GNU/Linux version 13.1.

\subsection{First example: a genus 2 hyperelliptic curve}

In this example, $p=3$ and $k$ is an algebraic closure of $\FF_p$. Consider the genus two curve $C$ defined over $k$ by $y^2=x^5 + x^2 + 1$.
Let us compute the group $\He^1(C,\ZZ/p^3\ZZ)$ using our algorithms, as well as the étale Galois covers of $C$ with group $\ZZ/p^3\ZZ$.

Choosing the non-special divisor given in affine coordinates by $D=(0,2)+(2,2)$, we get the following Hasse--Witt matrix for $C$:

\[ \begin{pmatrix}
1 & 0 \\
0 & 0
\end{pmatrix}
\]

Let $z\in k$ be a root of the primitive polynomial $X^9 + 2X^3 + 2X^2 + X + 1$.
The $\FF_3$-vector space $\He^1(C,\ZZ/3\ZZ)$ has dimension $1$ and is generated by the adele $\frac1x\delta_{(0,2)}$.
The $\ZZ/27\ZZ$-module $\He^1(C,\ZZ/27\ZZ)$ is thus also of dimension $1$, and is generated by the Witt vector of adeles:
\[r=\left(\frac1x\delta_{(0,2)},\frac{z^{6813}}x\delta_{(0,2)}+\frac2{x+1}\delta_{(2,2)},\frac{z^{912}}x\delta_{(0,2)}+\frac{z^{11355}}{x+1}\delta_{(2,2)}\right)\text.\]

The Witt vector $w\in W_3(k(C))$ given by:
\begin{align*}
w_0&=\frac{x^2 + 2}{x^3} + \frac{1}{x^3} y\\
w_1&=\frac{z^{757}x^5 + z^{12112}x^3 + z^{757}x^2 + z^{10598}}{x^3 + x^6} + \frac{z^{2271}x^3 + z^{757}}{x^3 + x^6} y\\
w_2&=\frac{z^{2736}x^5 + z^{13814}x^3 + z^{2736}x^2 + z^{12577}}{x^3 + x^6} + \frac{z^{3973}x^3 + z^{2736}}{x^3 + x^6} y
\end{align*}
satisfies $\wp(r)-w\in W_3(\AO_C)$, hence defines the only étale cyclic extension of degree $27$ of $C$. This extension is given by the function field extension $k(C)(t_0,t_1,t_2)$ of $k(C)$ where \[ \wp(t_0,t_1,t_2)=(w_0,w_1,w_2)\]
as Witt vectors, i.e.:
\begin{align*}
t_0^3 - t_0 &= w_0\\
t_1^3 - t_1 &= -t_0^7 + t_0^5 + w_1 \\
t_2^3-t_2 &= -t_1^7 + t_1^6 t_0^7 - t_1^6 t_0^5 + t_1^5 
      - 2 t_1^4 t_0^7 + 2 t_1^4 t_0^5 + t_1^3 t_0^{14}  \\
    &\quad - 2 t_1^3 t_0^{12} + t_1^3 t_0^{10} + t_1^2 t_0^7 
      - t_1^2 t_0^5 - t_1 t_0^{14} + 2 t_1 t_0^{12} 
      - t_1 t_0^{10} - t_0^{25}  \\
    &\quad + 4 t_0^{23} - 9 t_0^{21} + 13 t_0^{19} 
      - 13 t_0^{17} + 9 t_0^{15} - 4 t_0^{13} + t_0^{11} + w_2
\end{align*}

\subsection{Second example: a non-hyperelliptic genus 3 curve}

In this example, $p=5$ and $k$ is an algebraic closure of $\FF_p$.
Consider the smooth projective genus 3 Fermat curve $C$ over $k$ defined over by the affine equation $x^4+y^4-1=0$.
Choosing the non-special divisor given in affine coordinates by $D=(0,4)+(0,3)+(4,0)$, we get the following Hasse--Witt matrix for $C$:
\[ \begin{pmatrix}
1 & 1 & 2 \\
3 & 4 & 2 \\
0 & 0 & 3
\end{pmatrix}
\]

Let $z\in k$ be a root of the primitive polynomial $X^{20} + 3X^{12} + 4X^{10} + 3X^9 + 2X^8 + 3X^6 + 4X^3 + X + 2$.
The $\FF_5$-vector space $\He^1(C,\ZZ/5\ZZ)$ has dimension $3$ and a basis given by the adeles $(\frac1x\delta_{(0,4)},\frac1x\delta_{(0,3)},\frac1y\delta_{(4,0)})$.
The $\ZZ/25\ZZ$-module $\He^1(C,\ZZ/25\ZZ)$ is thus also of dimension $3$, and is generated by the following 2-truncated Witt vectors of adeles:
\begin{align*}
r_1=&\left(\frac1x\delta_{(0,4)},\frac{z^{18817350559709}}x\delta_{(0,4)}+\frac{z^{30738279514787}}x\delta_{(0,3)}\right)\\
r_2=&\left(\frac1x\delta_{(0,3)},\frac{z^{59376817603623}}x\delta_{(0,4)}+\frac{z^{30966580319415}}x\delta_{(0,3)}\right)\\
r_3=&\left(\frac1y\delta_{(4,0)},\right.\\
&\qquad\left.\frac{z^{65122179520073}}x\delta_{(0,4)}+\frac{z^{65122179520073}}x\delta_{(0,3)}+\frac{z^{29832849734571}}y\delta_{(4,0)}\right)
\end{align*}

The Witt vectors $w_1,w_2,w_3\in W_2(k(C))$ given by the coordinates below satisfy $r_i-\wp(w_i)\in W_2(\AO_C)$.

\begin{align*}
{w_1}_0=&\frac{z^{21855036417643} + z^{13907750447591}x^4 + z^{93380610148111}x^5}{x^5}  \\&\qquad+ \frac{z^{9934107462565} + z^{21855036417643}x^4}{x^5} y \\&\qquad\qquad + \frac{z^{53644180297851}}{x^5} y^2+ \frac{z^{89406967163085}}{x^5} y^3\displaybreak\\
{w_1}_1=&\frac{z^{54350322948285} + z^{46403036978233}x^4 + z^{30508465038129}x^5}{x^5}  \\&\qquad+ \frac{z^{42429393993207} + z^{54350322948285}x^4}{x^5} y \\&\qquad\qquad+ \frac{z^{86139466828493}}{x^5} y^2 + \frac{z^{26534822053103}}{x^5} y^3\\
{w_2}_0=&\frac{z^{57617823282877} + z^{1986821492513}x^4 + z^{33775965372721}x^5}{x^5}  \\&\qquad+ \frac{z^{37749608357747} + z^{57617823282877}x^4}{x^5} y \\&\qquad\qquad+ \frac{z^{5960464477539}}{x^5} y^2 + \frac{z^{17881393432617}}{x^5} y^3\\
{w_2}_1=&\frac{z^{38234149422989} + z^{58465508916555}x^4 + z^{14392291512833}x^5}{x^5}  \\&\qquad+ \frac{z^{8531841693973} + z^{38234149422989}x^4}{x^5} y \\&\qquad\qquad+ \frac{z^{85666175764403}}{x^5} y^2 + \frac{z^{10126705138537}}{x^5} y^3\\
{w_3}_0=&\frac{z^{61591466267903} + z^{21855036417643}x^4 + z^{37749608357747}x^5}{x^5}  \\&\qquad+ \frac{z^{37749608357747} + z^{57617823282877}x^4}{x^5} y \\&\qquad\qquad+\left( \frac{z^{61591466267903} + z^{85433324178059}x}{x^5 + 2x^6 + x^7} \right.\\&\qquad\qquad\qquad\left.+ \frac{z^{61591466267903}x^2 + z^{5960464477539}x^4}{x^5 + 2x^6 + x^7}\right) y^2\\
{w_3}_1=&\frac{z^{15666744768337} + z^{40934628215143}x^4 + z^{87192318498805}x^5}{x^5}  \\&\qquad+ \frac{z^{87192318498805} + z^{18729079066295}x^4}{x^5} y \\&\qquad\qquad+\left( \frac{z^{15666744768337} + z^{39508602678493}x}{x^5 + 2x^6 + x^7}\right.\\&\qquad\qquad\qquad\left.+ \frac{z^{15666744768337}x^2 + z^{6113101211919}x^4}{x^5 + 2x^6 + x^7}\right) y^2
\end{align*}

The corresponding Galois covers are given by the extensions $k(C)({t_i}_0,{t_i}_1)$, $i\in \{1,2,3\}$, where \[ \wp(t_i)=w_i\]
as Witt vectors, i.e.:

\begin{align*}
{{t_i}_0}^5-{t_i}_0&= {w_i}_0 \\
{{t_i}_1}^5-{t_i}_1&= {{t_i}_0}^{21} - 2{{t_i}_0}^{17} + 2{{t_i}_0}^{13} - {{t_i}_0}^9 + {w_i}_1
\end{align*}

\section{Application to étale cohomology computations}\label{sec:etale}

In this section, we use our main algorithm and adapt the ideas of \cite[§3]{cl} in order to compute the cohomology of locally constant étale sheaves of $\Zpn$-modules on smooth projective curves. In order to simplify the exposition, we will always suppose that we are given a non-special system of $g$ points on $X$ all defined over $\Fq$, and the corresponding basis of $\Hz^1(X,\OO_X)$.

\subsection{The cohomology of locally constant sheaves}

Let $k$ be an algebraically closed field of characteristic $p$. Let $X$ be a connected smooth projective curve over $k$, and $\LL$ be a locally constant sheaf of $\Zpn$-modules on $X$. Let $Y\to X$ be an étale Galois cover such that $\LL|_Y$ is a constant sheaf.

Denote by $\Ypn\to Y$ the maximal abelian étale cover of $Y$ with exponent $p^n$. The automorphism group $\Aut(\Ypn|Y)$ is the maximal abelian quotient of $\pi_1(Y)$ with exponent $p^n$, and there is a canonical isomorphism \[ \Aut(\Ypn|Y)\xrightarrow{\sim} \He^1(Y,\Zpn)^\vee.\]

Set $M=\He^0(Y,\LL|_Y)$. Given a group $G$ and a $G$-module $V$, we denote by $\Homcr(G,V)$ the abelian group of crossed homomorphisms $G\to V$, i.e. the maps $f\colon G\to V$ such that $\forall g,h\in G, f(gh)=f(g)+gf(h)$.

\begin{lem} The cohomology complex $\Re(X,\LL)$ is isomorphic, in the derived bounded category $\DD^b_c(\Zpn)$ of $\Zpn$-modules, to the following complex:
\[ M \longrightarrow \Homcr(\Aut(\Ypn|X),M).\]
\end{lem}

\begin{proof} 
Since $\LL|_Y$ is constant, the map \[ \He^1(Y,\LL|_Y)\to 
\He^1(\Ypn,\LL|_{\Ypn})\] is trivial by construction of $\Ypn$. 
By \cite[Proposition 3.1]{cl}, this implies that the truncation in degrees $\leq 1$ of the composite map \[ \RG(\Aut(\Ypn|X),M)\to \RG(\pi_1(X),M) \to \Re(X,\LL)\]
is an isomorphism in $\DD^b_c(\Zpn)$. The truncation of the complex on the left-hand side is exactly the complex considered in the statement. As the cohomology of $\LL$ on $X$ is concentrated in degrees 0 and 1 \cite[VI, Remark 1.5.(b)]{milneEC}, this concludes the proof. 
\end{proof}
Hence, computing the étale cohomology complex of $\LL$ boils down to computing the automorphism group $\Aut(\Ypn|X)$ with its action on $M$; the remaining group cohomology computations are just linear algebra.

\subsection{Computing the automorphism group}

Consider an étale Galois covering $Y\to X$, corresponding to a function field extension $K_Y/K_X$.
Here is how to find a preimage $\sigma\in\Aut(\Ypn|X)$ of an automorphism $\tau\in\Aut(Y|X)$.

Recall that by Lemma \ref{lem:H1Wn}, there is an isomorphism \[ \Hz^1(Y,\Wn(\OO_Y))\overset{\sim}{\longrightarrow} \frac{W_n(\AA_Y)}{W_n(\AO_Y)+W_n(K_Y)}\]
and by Theorem \ref{th:asw}, the group $\He^1(Y,\Zpn)$ is isomorphic to its subgroup $\Hz^1(Y,\Wn(\OO_Y))^{F=\id}$ of Frobenius-invariant elements.
Consider a basis of the free $\Zpn$-module $\He^1(Y,\Zpn)$ given by elements $r^{(1)},\dots,r^{(s)}\in W_n(\AA_Y)$. For all $i\in \{1\dots s\}$, there is a Witt vector $f^{(i)}\in W_n(K_Y)$ such that \[ F(r^{(i)})-r^{(i)}\equiv f^{(i)} \mod W_n(\AO_Y).\]
The function field $K_{\Ypn}$ of $\Ypn$ satifies $K_{\Ypn}=K_Y(t_{0}^{(1)},\dots,t_{n-1}^{(1)},\dots,t_{n-1}^{(s)})$ where for all $i\in \{1\dots s\}$, $F(t^{(i)})-t^{(i)}=f^{(i)}$ in $W_n(K_{\Ypn})$.

For all $i\in\{1\dots s\}$, denote by $\tau_{ij}\in\Zpn$ the coordinates of $\tau^\star r^{(i)}$ in the basis $r^{(1)},\dots,r^{(s)}$. There is an element $h^{(i)}\in W_n(K_Y)$ such that 
\[ \tau^\star r^{(i)}\equiv \sum_{j=1}^s \tau_{ij}r^{(j)}+h^{(i)} \mod W_n(\AO_Y).\]
Applying $\wp=F-\id$ to this equality, we obtain
\[ \tau^\star f^{(i)} \equiv \sum_{j=1}^s \tau_{ij}f^{(j)} +\wp(h^{(i)}) \mod W_n(\AO_Y).\]
This means that \[ \tau^\star f^{(i)} - \left(\sum_{j=1}^s \tau_{ij}f^{(j)} +\wp(h^{(i)})\right) \in W_n(\AO_X)\cap W_n(K_Y)=W_n(k).\]
Let $u^{(i)}\in W_n(k)$ be this element. Since $k$ is algebraically closed, there exists $v^{(i)}\in W_n(k)$ such that $\wp(v^{(i)})=u^{(i)}$. 

\begin{lem} For all $i\in \{1\dots s\}$ and $j\in\{0\dots n-1\}$, denote by $w_{j}^{(i)}$ the $j$-th coordinate of the Witt vector $\tau_{i1}t^{(1)}+\dots+\tau_{is}t^{(s)}+h^{(i)}+v^{(i)}\in W_n(K_{\Ypn})$.
The endomorphism $\sigma$ of $\Ypn$ defined by $\sigma^\star|_{K_Y}=\tau^\star$ and, for all $i\in\{1\dots s\}$ and $j\in \{0\dots n-1\}$,
\[ \sigma^\star (t_{j}^{(i)})= w_{j}^{(i)}\]
is a preimage of $\tau$ in $\Aut(\Ypn|X)$. 
\end{lem}

\begin{proof} Denote by $\sigma^\star t^{(i)}$ the Witt vector $(\sigma^\star t_{0}^{(i)},\dots,\sigma^\star t_{n-1}^{(i)})$. We simply have to prove the following equality in $W_n(K_{\Ypn})$, for all $i\in \{1\dots s\}$: \[\wp(\sigma^\star t^{(i)})=\tau^\star f^{(i)}.\] 
We know that \[ \sigma^\star t^{(i)}=\sum_{j=1}^s \tau_{ij}t^{(j)}+h^{(i)}+v^{(i)}.\]
Hence \begin{align*} \wp(\sigma^\star t^{(i)}) &= \sum_{j=1}^s \tau_{ij}\wp(t^{(j)})+\wp(h^{(i)})+\wp(v^{(i)})\\
&= \sum_{j=1}^s \tau_{ij}f^{(j)}+\wp(h^{(i)})+\wp(v^{(i)})\\
&= \tau^\star f^{(i)}.
\end{align*}
\end{proof}

\begin{algorithm}[H]\label{alg:computeautomorphisms}
\SetAlgoLined  
\caption{\textsc{ComputeAutomorphisms}}
\KwData{Etale Galois cover $Y\to X$ of curves, given by function field extension $K_Y/K_X$

Automorphism $\tau\in\Aut(Y|X)$

Basis $B=(r^{(1)},\dots,r^{(s)})$ of $\He^1(X,\Zpn)$

$f^{(1)},\dots,f^{(s)}\in W_n(K_Y)$ such that $F(r^{(i)})-r^{(i)}\equiv f^{(i)}\mod W_n(\AO_Y)$

Function field $K_{\Ypn}=K_Y(t_{j}^{(i)})_{\substack{1\leq i\leq s \\ 0\leq j< n}}$ with $\wp (t^{(i)})=f^{(i)}$}
\KwResult{Preimage $\sigma$ of $\tau$ in $\Aut(\Ypn|X)$}
\hrulefill

\For{$i=1\dots s$}{
Compute $\tau^\star r^{(i)}$

$((\tau_{ij})_{j},h^{(i)})\coloneqq\textsc{CoordinatesInBasisWitt}(S,\tau^\star r^{(i)},r^{(1)},\dots,r^{(s)})$

Compute $u^{(i)}=\tau^\star f^{(i)}-\sum_j \tau_{ij}f^{(j)}-\wp(h^{(i)})\in W_n(k)$

Compute $v^{(i)}\in W_n(k)$ such that $\wp(v^{(i)})=u^{(i)}$

Set $w_{j}^{(i)}=\sum_j \tau_{ij}t^{(j)}+h^{(i)}+v^{(i)}$}
\Return{$\sigma\colon t_{j}^{(i)}\mapsto w_{j}^{(i)}$}
\end{algorithm}

\begin{lem}\label{lem:compautos} Suppose $X$ is defined over $\Fq$ and given by a plane projective model of degree $d_X$ with ordinary singularities. Suppose $X$ admits a non-special system of points all defined over $\Fq$. Then
Algorithm \ref{alg:computeautomorphisms} computes a preimage of $\tau$ in $\Aut(\Ypn|X)$ in \[\Poly(p^{n^2},g,d_X) \] operations in the field of definition of the $r^{(i)}$ and $f^{(i)}$.  
\end{lem}
\begin{proof}
Looking for a preimage under $\wp\colon W_n(k)\to W_n(k)$ is done by moving to a field extension (possibly of degree $p^n$) and then finding the roots of linearised polynomials of $p$-degree at most $n$: this is polynomial-time in $p^n$.
Hence the complexity is dominated by the $s$ calls to \textsc{CoordinatesInBasisWitt}, which are polynomial-time in $p^{n^2}$ by
Lemma \ref{lem:compcoordsinbasiswitt}.
\end{proof}

\subsection{The algorithm}

\begin{algorithm}[H]\label{alg:computecohomology}
\SetAlgoLined  
\caption{\textsc{ComputeCohomology}}
\KwData{\'Etale Galois cover $Y\to X$ of curves, given by function field extension $K_Y/K_X$

Generators $\tau_1,\dots,\tau_r$ of $\Aut(Y|X)$

Basis $B=(r^{(1)},\dots,r^{(s)})$ of $\He^1(X,\Zpn)$

Functions $f_1,\dots,f_s\in K_Y$

Locally constant sheaf $\LL$ given by a $(\Zpn)[\Aut(Y|X)]$-module $M$}
\KwResult{A complex isomorphic to $\Re(X,\LL)$ in $\DD^b_c(\Zpn)$}
\hrulefill

$K_{\Ypn}\coloneqq\textsc{ComputeMaximalCover}(K_Y,n)$

\For{$i=1\dots r$}{$\sigma_i\coloneqq\textsc{ComputeAutomorphisms}(K_Y,\tau_i,B,f_1,\dots,f_s,K_{\Ypn})$}
Compute the group law of $\Aut(\Ypn|X)$

Compute $\Homcr(\Aut(\Ypn|X),M)$ using linear algebra

\Return $M\to \Homcr(\Aut(\Ypn|X),M)$
\end{algorithm}

\etalecomp*
\begin{proof}
Denote by $g_Y$ the genus of $Y$. By the Riemann--Hurwitz formula, $g_Y=O(g[K_y:K_X])$. The complexity of computing the maximal cover $\Ypn$ is $\Poly(q^{g_Y^2},p^{n^2},g,d_X)$ by Theorem \ref{th:maincomp}. The cover $\Ypn$ is defined by equations with coefficients in a field extension $\FF_Q$, with $\log_q(Q)=\tilde O(nq^{g^2})$ by Lemma \ref{lem:compH1}.
By Lemma \ref{lem:compautos}, each of the $[K_Y:K_X]$ calls to \textsc{ComputeAutomorphisms} takes $\Poly(q^{g_Y^2},p^{n^2},d_X)$ operations in $\FF_Q$. 
Since $\Aut(\Ypn|X)$ has order $p^n[K_y:K_X]$, this also dominates the complexity of computing the group law of $\Aut(\Ypn|X)$.  
Computing the cohomology complex of $M$ is linear algebra over $\Fp$, and requires a number of $\Fp$-operations which is polynomial in $m$ and $|\Aut(\Ypn|X)|$.
\end{proof}

\bibliographystyle{alphaurl}
\bibliography{effectiveasw.bib}

\end{document}